\documentclass{article}
\usepackage{graphicx} 
\usepackage{latexsym,color,amsmath,amsthm,amssymb,amscd,amsfonts}
\usepackage{scalefnt}
\usepackage[font=small,labelfont=bf]{caption}
\usepackage{float}
\usepackage{graphicx}
\usepackage{amsmath}
\usepackage{amsmath}
\usepackage{relsize}
\usepackage{mathrsfs}
\usepackage{upgreek}
\usepackage{sectsty}
\usepackage{bbm}

\newtheoremstyle{nonum}{}{}{\itshape}{}{\bfseries}{.}{ }{\thmnote{#3}}

\newtheorem{theorem}{Theorem}
\newtheorem{lemma}[theorem]{Lemma}

\usepackage{bbm}

\def\eps{{\varepsilon}}

\def\conv{{\rm conv}}

\title{A very short proof of the Figiel-Lindenstrauss-Milman theorem}
\author{Tomer Milo}
\date{}

\begin{document}

\maketitle

\noindent The goal of this note is to provide a new proof for the following theorem, originally proved in \cite{FLM} by Figiel, Lindenstrauss and Milman:

\begin{theorem}\label{main theorem} There exists an absolute constant $c>0$ such that for every $n \in \mathbb{N}$ and for every convex polytope $P \subset \mathbb{R}^{n}$ which is symmetric about the origin (that is, $P= -P$) we have:
\[ \log |V| \cdot \log|\mathcal{F}| \geq cn  \]
Where $V$ and $\mathcal{F}$ are the sets of vertices ($0$-dimensional faces) and facets ($(n-1)$-dimensional faces) of $P$, respectively. Moreover, the same estimate holds for a general polytope $P$ satisfying \newline $B_n \subset P \subset \sqrt{n}B_n$, where $B_n$ is the Euclidean unit ball in $\mathbb{R}^n$, and one can actually take $c= c_n = \frac{1}{9} + o(1)$ as $n \to \infty$, where $\log$ is the standard logarithm to base $e$.
\end{theorem}

The fact that the theorem holds for general polytopes $P$ satisfying $B_n \subset P \subset \sqrt{n}B_n$ implies that it also holds for every origin-symmetric polytope; this is a consequence of John's theorem, which we state now. A good exposition on John's theorem can be found in \cite[Section 2]{AGA}.

\begin{theorem}\label{theorem: john}(John)
Let $K \subset \mathbb{R}^n$ be such that $K = - K$. There exists a linear transformation $T$ such that $B_n \subset T(K) \subset \sqrt{n}B_n$.
\end{theorem}

Indeed, for an origin symmetric polytope $P$ we may apply an appropriate linear transformation $T$ so that $B_n \subset T(P) \subset \sqrt{n}B_n$. Since $T$ does not change the combinatorial structure of $P$, the symmetric part of Theorem \ref{main theorem} will follow.
\\
Theorem \ref{main theorem} is classically proved as a consequence of Dvoretzky's theorem; see \cite[section 5]{AGA} for an exposition of Milman's classical proof and its applications. Here we provide a new proof avoiding Dvoretzky's theorem completely (as well as the Dvoretzky-Rogers lemma), and also avoiding the classical concentration inequality for Lipschitz functions on the Euclidean sphere $S^{n-1} \subset \mathbb{R}^n$, which is proved via the spherical isoperimetric inequality due to Levy. We shall only need the following simple concentration result, for which an elementary proof can be found in \cite[Theorem 3.1.5]{AGA}:

\begin{lemma}\label{concentration of measure} Let $\sigma$ denote the unique rotational invariant probability measure on the Euclidean unit sphere $S^{n-1}$. For every $u \in S^{n-1}$ and every $\eps \in (0,1)$,

\[ \sigma(\{\theta \in S^{n-1}: \langle \theta, u \rangle < \eps \}) > 1 - e^{-\frac{n\eps^2}{2}}. \]

\end{lemma}

We prove Theorem \ref{main theorem} for a general polytope $ P = \conv(V) \subset \mathbb{R}^n$  satisfying $B_n \subset P \subset \sqrt{n}B_n$. Throughout the proof, the symbols $c,C,c_1$ etc. will denote positive absolute constants whose value may be different in different lines. Let $\|x\|_P = \inf\{t: x \in tP\}$ denote the gauge function of $P$, and let $h_{P}(x) := \max_{y\in P} \langle x, y\rangle = \max_{v \in V} \langle x, v \rangle $. We introduce the following two parameters:
\[ M(P) = \int_{S^{n-1}} \|\theta\|_{P}d\sigma(\theta),  \quad  M^{*}(P) = \int_{S^{n-1}}h_{P}(\theta)d\sigma \]
and note the simple relationship $M^*(P) = M(P^{\circ})$, where $P^{\circ} := \{ x \in \mathbb{R}^n: \max_{y \in P} \langle x , y \rangle \leq 1 \}$ is the dual body of $P$. Our main ingredient of the proof is the following well known bound on $M^*(P)$ in terms of the vertices of $P$.

\begin{lemma}\label{lemma: mean width bound}
    Let $P \subset RB_n$ be a polytope with vertex set $V$ such that $\log |V| < \frac{n}{3}$. Then
    \[ M^*(P) := \int_{S^{n-1}} h_{P}d\sigma \leq CR \sqrt{\frac{\log |V|}{n}} \]
    Where $C = C_n = \sqrt{3} + o(1)$ as $n \to \infty$.
\end{lemma}
\begin{proof}
We write $h_{P}(\theta) = \max_{v \in V} \langle v, \theta \rangle$. Let $B_t = \{\theta \in S^{n-1}: h_{P}(\theta) \leq t \}$. Using Lemma \ref{concentration of measure}  and a union bound again gives:
    \[ M^*(P) = \int_{S^{n-1}}\max_{v \in V} \langle v, \theta \rangle d\sigma = \int_{B_t}\max_{v \in V} \langle v, \theta \rangle d\sigma  + \int_{S^{n-1} \setminus B_t }\max_{v \in V} \langle v, \theta \rangle d\sigma \leq t + R|V|e^{-\frac{1}{2}n(\frac{t}{R})^2}.
    \] 
Choosing $t = R\sqrt{\frac{3\log|V|}{n}}$ and using the fact that $|V| > n $ gives
\[ M^{*}(P) \leq  R\left( \sqrt{\frac{3\log|V|}{n}} + \frac{1}{\sqrt{|V|}}\right) \leq \left(\sqrt{3} + o(1)\right) R \sqrt{\frac{\log|V|}{n}}, \]

\noindent where the $o(1)$ term has a decay rate of at most $\frac{1}{\sqrt{\log n}}$.

\end{proof}

We apply Lemma \ref{lemma: mean width bound} to the dual polytope $P^{\circ}$. We use the following facts: $M^*(P) = M(P^\circ)$, $R(P) = \frac{1}{r(P^{\circ})}$ where $r = r(P^\circ)$ is the maximal $r$ for which $rB_n \subset P^{\circ}$, and that the number of facets $\mathcal{F}$ of $P$ is the same as the number of vertices of $P^{\circ}$. We get:
\begin{equation}
 M(P) \leq \left(\sqrt{3} + o(1)\right)\frac{1}{r}\sqrt{\frac{\log|\mathcal{F}|}{n}}.    
\end{equation}

\begin{proof}[Proof of Theorem \ref{main theorem}]
Combining Lemma \ref{lemma: mean width bound} with (1) (note that if the condition $\log|V| < \frac{n}{3}$ of Lemma \ref{lemma: mean width bound} is not satisfied, there is nothing to prove), using the trivial bound $M(P)M^*(P) \geq 1$ and our assumption that $\frac{r}{R} \geq \frac{1}{\sqrt{n}}$ we finally get:

\[ \log|V| \cdot \log|\mathcal{F}| \geq \left(\frac{1}{9} + o(1)\right) n^2 \left( \frac{r}{R} \right)^2 (M(P)M^*(P))^2 \geq \left(\frac{1}{9} + o(1)\right)n \]

\end{proof}

\bibliographystyle{amsplain}

\begin{thebibliography}{99}

\bibitem{AGA}
Artstein-Avidan S., Giannopoulos A., Milman V.D.,
{\em Asymptotic Geometric Analysis, Part I}. Mathematical Surveys and Monographs,  volume 202, American Mathematical Society, Providence, RI. (2015).


\bibitem{FLM}
Figiel, T., Lindenstrauss, J.,  Milman, V.D. {\em The dimension of almost spherical sections of convex bodies}. Acta Math. 139, 53–94 (1977). https://doi.org/10.1007/BF02392234

\end{thebibliography}

\end{document}